\documentclass{article}
\usepackage[tbtags]{amsmath}
\usepackage{amsfonts,amssymb,mathrsfs,amscd}
\usepackage{graphicx}

\newtheorem{thm}{Theorem}[section]
\newtheorem{prop}{Proposition}[section]
\newtheorem{lem}{Lemma}[section]
\newtheorem{cor}{Corollary}[section]

\newenvironment{dfn}{\trivlist \item[\hskip \labelsep{\bf Definition.}]}%
{\endtrivlist}
\newenvironment{proof}{\trivlist \item[\hskip \labelsep{\bf
Proof.}]}{\endtrivlist}
\newenvironment{Proof}[1]{\trivlist \item[\hskip \labelsep{\bf #1.}]}%
{\endtrivlist}

\catcode`@=11
 \def\.{.\spacefactor\@m}
\catcode`@=12
\def\<{\langle}
\def\>{\rangle}
\def\:{\colon}
\def\ss{\subset}
\def\sp{\supset}
\def\c{\circ}

\def\R{{\Bbb R}}

\def\bN{{\bar N}}

\def\g{{\gamma}}
\def\G{\Gamma}
\def\dl{\delta}
\def\e{\varepsilon}
\def\l{\lambda}
\def\om{\omega}
\def\r{\rho}

\def\MST{{\operatorname{MST}}}
\def\SMT{{\operatorname{SMT}}}
\def\len{{\operatorname{len}}}
\def\diag{\operatorname{diag}}
\def\ch{\operatorname{cosh}}
\def\sh{\operatorname{sinh}}
\def\arcch{\operatorname{arccosh}}

\begin{document}

 \title{Steiner Ratio for Riemannian Manifolds}
 \author{D.~Cieslik, A.~O.~Ivanov, A.~A.~Tuzhilin}
 \date{}
 \maketitle
 \begin {abstract}
For a metric space $(X,\r)$ and any finite subset $N\ss X$ by
$\r(\SMT_N)$ and $\r(\MST_N)$ we denote respectively the lengths of a
Steiner minimal tree and a minimal spanning tree with the boundary $N$.
The {\em Steiner ratio $m(X,\r)$} of the metric space
is the value $\inf_{\{N:N\ss X\}}\frac{\r(\SMT_N)}{\r(\MST_N)}$.
In this paper we prove the following results describing the Steiner ratio
of some manifolds:

(1) the Steiner ratio of an arbitrary $n$-dimensional connected Riemannian
manifold $M$ does not exceed the Steiner ratio of $\R^n$;

(2) the Steiner ratio of the base of a locally isometric covering is
more or equal than the Steiner ratio of the total space;

(3) the Steiner ratio of a flat two-dimensional torus, a flat Klein
bottle, a projective plain having constant positive curvature is equal to
$\sqrt3/2$;

(4) the Steiner ratio of the curvature $-1$ Lobachevsky space does not
exceed $3/4$;

(5) the Steiner ratio of an arbitrary surface of constant negative
curvature $-1$ is strictly less than $\sqrt3/2$.

\smallskip
{\bf Keywords:} Steiner minimal tree (SMT), minimal spanning tree (MST),
the Steiner problem, the Steiner ratio, metric space, Riemannian manifold.
 \end {abstract}

 \section {Introduction and main results}
 \markright {Introduction and main results.}

{\let\thefootnote\relax
 \footnote{A.~Ivanov and A.~Tuzhilin were partially supported by
RFBR (grants 96--15--96142 and 98--01--00240) and INTAS (grant 97--0808).
 }
}

\setcounter{footnote}{0}

Let $V$ be an arbitrary finite set. Recall that a {\em graph $G$ on $V$}
is the pair $(V,E)$, where $E$ is a finite set that consists of some pairs
of elements from $V$. Notice that $E$ can contain several copies of some
pair, and also $E$ can contain the pairs of the form $\{v,v\}$, where
$v\in V$. Elements from $V$ are called {\em vertices of $G$}, and the
elements from $E$ are called {\em edges of $G$}. The edges of the form
$(v,v)$ are called {\em loops}, and if $E$ contains several copies of an
edge $e=\{v,v'\}\in E$, then the edge $e$ is called a {\em multiple edge}.
For a given graph $G$ we denote the set of all its vertices by $V(G)$, and
the set of all its edges by $E(G)$.  For convenience, we shall often
denote the edge $e=\{x,y\}\in E(G)$ by $xy$.

Sometimes it is  useful to consider graphs as topological spaces glued from
segments each of which corresponds to an edge of the graph. Such graphs are
called {\em topological graphs}. A continuous mapping $\G$ from a
topological graph $G$ into a topological space is called a {\em network};
the topological graph $G$, and also the standard graph corresponding to
$G$, are called the {\em type\/} of $\G$ or the {\em topology\/} of $\G$.
Thus, the edges of a network are continuous curves in the ambient space.
Moreover, all the terminology of the Topological Spaces Theory is
transferred to the topological graphs and networks. If the ambient space is
a smooth manifold, then a network in such space is called {\em smooth\/}
({\em piecewise-smooth\/}), if all its edges are smooth
(piecewise-smooth).

A graph $G$ is called {\em weighted} if it is given a non-negative
function $\om\:E(G)\to\R$ called the {\em weight function}. The number
$\om(e)$ is called the {\em weight of the edge $e\in E(G)$}. The sum of
the weights over all edges of $G$ is called the {\em weight of the graph
$G$} and it is denoted by $\om(G)$. If $G$ is a connected weighted graph,
then the set of all connected spanning subgraphs of $G$ having the least
weight contains a tree. Each such tree is called a {\em minimal spanning
tree\/} and is denoted by $\MST_G$.  Notice that if all the weights are
strictly greater than zero, then any connected spanning subgraph of $G$ of
the least weight is a tree.

Let $X$ be a set, $\r$ be a metric on $X$, and $N$ be an arbitrary finite
subset of $X$. Let $G$ be a complete graph on $N$. The metric $\r$
generates the weight function that assigns to each edge $xy\in E(G)$ the
number $\r(x,y)$. This weight function will be denoted by the same letter
$\r$. Minimal spanning tree in the graph $G$ is denoted by $\MST_N$. A
{\em minimal Steiner tree on the set $N$} or a {\em minimal Steiner tree
spanning the set $N$} is defined to be a tree $\G$, $N\ss V(\G)$, such
that
 \begin {equation}\label{eq:SMT}
\r(\G)=\inf_{\{\bN:\bN\subset N\}}\r(\MST_\bN),
 \end {equation}
where the least upper bound is taken over all finite subsets $\bN$ in $X$ that
contain $N$. A minimal Steiner tree on the set $N$ is denoted by $\SMT_N$.

Note that, generally speaking, an $\SMT_N$ exists not for any $N$ (one of
the reasons of that can be the incompleteness of the metric space
$(X,\r)$). Nevertheless, the greatest lower bound from the definition of
$\SMT_N$ does always exist.  {\bf In what follows, the greatest lower
bound from~(\ref{eq:SMT}) is always denoted by $\r(\SMT_N)$, irrespective
of the existence of $\SMT_N$.}

The novelty of Steiner's Problem is that new points, the
Steiner points, may be introduced so that an interconnecting network
of all these points will be shorter. Given a set of points, it is
a priori unclear how many Steiner points one has to add in order to
construct an $\SMT$. Whereas Steiner's Problem is very
hard as well in combinatorial as in computational sense, the determination
of a Minimum Spanning Tree is simple. Consequently, we are interested in

 \begin {dfn}
The {\em Steiner ratio $m(X,\r)$ of a metric space $(X,\r)$} is defined as
the following value:
 $$
m(X,\r)=\inf_{\{N:N\ss X\}}\frac{\r(\SMT_N)}{\r(\MST_N)}.
 $$
 \end {dfn}

It is clear that the Steiner ratio of any metric space is always a
nonnegative number with $m(X,\r)\le1$. The Steiner ratio is a parameter of
the considered space and describes the approximation ratio for Steiner's
Problem. The quantity $m(X,\r)\cdot\r(\MST_N)$ would be a convenient lower
bound for the length of an $\SMT$ for $N$ in $(X,\r)$; that means, roughly
speaking, $m(X,\r)$ says how much the total length of an $\MST$ can be
decreased by allowing Steiner points.

 \begin{prop}[E.F.Moore, in \cite{Gilbert:68}] For the Steiner ratio
of any metric space $(X,\r)$ the inequalities
 $$
\frac12\leq m(X,\r)\leq1
 $$
hold.
 \end{prop}

It is also shown that these inequalities are the best possible ones over
the class of metric spaces.%
 \footnote{And, indeed, there are metric spaces with Steiner ratios equals
$1$ and equals $0.5$.
 }

As an introductory example consider three points which form the nodes of
an equilateral triangle of unit side length in the Euclidean plane. An
$\MST$ for these points has length $2$. An $\SMT$ uses one Steiner point,
which is uniquely determined by the condition that the three angles at
this point are equal, and consequently equal $120^\c$. Consequently, we
find the length of the $\SMT$ in $3\cdot\sqrt{1/3}=\sqrt{3}$. So we have
an upper bound for the Steiner ratio of the Euclidean plane:
 \begin{equation}\label{gp}
m\leq\frac{\sqrt3}2 = 0.86602\ldots .
 \end{equation}
A long-standing conjecture, given by Gilbert and Pollak
\cite{Gilbert:68} in 1968, said that in the above inequality equality
holds.  This was the most important conjecture in the area of Steiner's
Problem in the following years.  Finally, in 1990, Du and Hwang
\cite{DuHwang} created many new methods and succeeded in proving the
Gilbert-Pollak conjecture completely: The Steiner Ratio of the Euclidean
plane equals $\sqrt3/2=0.86602\ldots$.%
 \footnote{This mathematical fact went in The New York Times, October 30,
1990 under the title "Solution to Old Puzzle: How Short a Shortcut?"
 }

For each dimension $n>2$, at present, exact values for the Steiner ratios
of the Euclidean spaces are not yet known. In particular, this is true for
$n=3$.

$\SMT$'s have been the subject of extensive investigations during the past
30 years or so. Most of this research has dealt with the Euclidean metric,
with much of the remaining work concerned with the ${\cal L}_1$-metric, or
more generally, the usual ${\cal L}_p$-metric or with two-dimensional
Banach spaces. An overview for the Steiner ratios of these metric spaces
is given in \cite{Cieslik1}.

The first results concerning the Steiner ratios of Riemannian manifolds
different from Euclidean spaces were obtained by J.~H.~Rubinstein and
J.~F.~Weng in 1997, see~\cite{Rubin}. They have shown that the Steiner
ratio for the standard two-dimensional spheres is the same as for the
Euclidean plane, that is, $\sqrt3/2$.

Now we list the main results of the present article. These results were
obtained by means of the technique worked out in~\cite{Cieslik1},
\cite{IT6}, and~\cite{ITbookWP}. Let us mention that in \cite{IT6}
and~\cite{ITbookWP} the authors investigate so called local minimal
networks which turn out to be useful in the subject.

 \begin {thm}\label{th:riem}
The Steiner ratio of an arbitrary $n$-dimensional connected Riemannian
manifold $M$ does not exceed the Steiner ratio of $\R^n$.
 \end {thm}

 \begin {thm}\label{th:bundle}
Let $\pi\:W\to M$ be a locally isometric covering of connected Riemannian
manifolds. Then the Steiner ratio of the base $M$ of the covering is more
or equal than the Steiner ratio of the total space $W$.
 \end {thm}

 \begin {cor}\label{cor:2-dim}
The Steiner ratio for a flat two-dimensional torus, a flat Klein bottle, a
projective plain having constant positive curvature is equal to
$\sqrt3/2$.
 \end {cor}

Thus, taking into account the results of J.~H.~Rubinstein and
J.~F.~Weng~\cite{Rubin}, the Steiner ratio is computed now for all closed
surfaces having non-negative curvature.

 \begin {thm}\label{th:hyper}
The Steiner ratio of the curvature $-1$ Lobachevsky space does not
exceed $3/4$.
 \end {thm}

 \begin {thm}\label{th:hypersurf}
The Steiner ratio of an arbitrary surface of constant negative curvature
$-1$ is strictly less than $\sqrt3/2$.
 \end {thm}

The authors want to thank the Ernst--Moritz--Arndt University of
Greifswald for the opportunity to work together in Greifswald in March
2000. A.~Ivanov and A.~Tuzhilin are grateful to academic A.~T.~Fomenko for
his kind interest to our work.

 \section {Proofs of the theorems}
 \markright {Proofs of the theorems.}
In the present section we give the proofs of the theorems stated above.

We need the following two Lemmas proved in~\cite{Cieslik1} (notice
that Lemma~\ref{lem:equiv_metr} is proved in~\cite{Cieslik1} for the
case of normalized spaces only, but the proof in the general case of
metric spaces is just the same.)

 \begin {lem}\label{lem:equiv_metr}
Let $X$ be a set, and $\r_1$ and $\r_2$ be two metrics on $X$. We assume
that for some numbers $c_2\ge c_1>0$ and for arbitrary points $x$ and $y$
from $X$ the following inequality holds:  $c_1\r_2(x,y)\le\r_1(x,y)\le
c_2\r_2(x,y)$. Then
 $$
\frac{c_1}{c_2}m(X,\r_2)\le m(X,\r_1)\le\frac{c_2}{c_1}m(X,\r_2).
 $$
 \end {lem}

 \begin {lem}\label{lem:subspace}
Let $(X,\r)$ be a metric space, and $Y\ss X$ be some its subspace. Then
 $$
m(Y,\r)\ge m(X,\r).
 $$
 \end {lem}

The following Proposition holds.

 \begin {prop}\label{prop:bundle}
Let $f\: X\to Y$ be some mapping of a metric space $(X,\r_X)$ onto a
metric space $(Y,\r_Y)$. We assume that $f$ does not increase the
distances, that is, for arbitrary points $x$ and  $y$ from $X$ the
following inequality holds:
 $$
\r_Y\bigl(f(x),f(y)\bigr)\le\r_X(x,y).
 $$
Then for arbitrary finite set $N\ss Y$ we have:
 $$
\r_X\bigl(\MST_N\bigr)\ge\r_Y\bigl(\MST_{f(N)}\bigr),\ \
\r_X\bigl(\SMT_N\bigr)\ge\r_Y\bigl(\SMT_{f(N)}\bigr).
 $$
 \end {prop}

 \begin {proof}
Let $G$ be an arbitrary connected graph constructed on $N$. We consider
two weight functions on $G$ defined on the edges $xy$ of $G$ as follows:
$\r_X(xy)=\r_X(x,y)$, and $\om_Y(xy)=\r_Y\bigl(f(x),f(y)\bigr)$.  Since
$f$ does not increase the distances, then $\r_X(G)\ge\om_Y(G)$.

Let $G'$ be a graph on $N'=f(N)$, such that the number of edges joining
the vertices $x'$ and $y'$ from $N'=V(G')$ is equal to the number of edges
from $G$ joining the vertices from $f^{-1}(x')\cap N$ with the vertices
from $f^{-1}(y')\cap N$. It is clear that $G'$ is connected, and
$\r_Y(G')=\om_Y(G)$.

Conversely, it is easy to see that for an arbitrary connected graph $G'$
constructed on $f(N)$ there exists a connected graph $G_X$ on $N$, such
that $\r_Y(G')=\om_Y(G_X)$. (To construct $G_X$ it suffices to span each
set $N\cap f^{-1}(x')$, $x'\in N'$, by a connected graph, and then to join
each pair of the constructed graphs corresponding to some adjacent
vertices $G'$ by $k$ edges, where $k$ is the multiplicity of the
corresponding edge in $G'$).  Therefore,
 \begin {multline*}
\r_X(\MST_N)=\inf_{\{G:V(G)=N\}}\r_X(G)\ge\inf_{\{G:V(G)=N\}}\om_Y(G)=\\
\inf_{\{G':V(G')=f(N)\}}\r_Y(G')=
\r_Y\bigl(\MST_{f(N)}\bigr).
 \end {multline*}
Thereby, the first inequality is proved.

Now let us prove the second inequality. We have:
 \begin {multline*}
\r_X(\SMT_N)=\inf_{\{\bN:\bN\sp N\}}\r_X(\MST_\bN)\ge
             \inf_{\{\bN:\bN\sp N\}}\r_Y(\MST_{f(\bN)})\ge\\
\inf_{\{\bN':\bN'\sp f(N)\}}\r_Y(\MST_{\bN'})=\r_Y(\SMT_{f(N)}).
 \end {multline*}
The proof is complete.
 \end {proof}

 \begin {prop}\label{prop:project}
Let $f\: X\to Y$ be a mapping of a metric space $(X,\r_X)$ to a metric
space $(Y,\r_Y)$, and let $f$ do not increase the distances. We assume
that for each finite subset $N'\ss Y$ there exists a finite subset $N\ss
X$, such that $f(N)=N'$ and
 \begin {equation}\label{eq:bundle}
\r_X(\SMT_N)\le\r_Y(\SMT_{N'}).
 \end {equation}
Then
 $$
m(X,\r_X)\le m(Y,\r_Y).
 $$
 \end {prop}

 \begin {proof}
Let $N\ss X$ be an arbitrary finite set. We have
 \begin {multline*}
m(X,\r_X)=\inf_{\{N:N\ss X\}}\frac{\r_X(\SMT_N)}{\r_X(\MST_N)}=\\
\inf_{\{N':N'\ss Y\}}\inf_{\{N:f(N)=N'\}}
                  \frac{\r_X(\SMT_N)}{\r_X(\MST_N)}\le\\
\inf_{\{N':N'\ss Y\}}\frac{\r_Y(\SMT_{N'})}{\r_Y(\MST_{N'})}=m(Y,\r_Y),
 \end {multline*}
where the inequality follows from both condition~(\ref{eq:bundle})
and the first inequality of Proposition~\ref{prop:bundle}.
The proof is complete.
 \end {proof}

Proposition~\ref{prop:project} can be slightly reinforced as follows.

 \begin {prop}\label{prop:project1}
Let $f\: X\to Y$ be a mapping of a metric space $(X,\r_X)$ to a metric
space $(Y,\r_Y)$, and let $f$ do not increase the distances. We assume
that for each finite subset $N'\ss Y$ the following inequality holds:
 \begin {equation}\label{eq:bundle-new}
\inf_{\{N:f(N)=N'\}}\r_X(\SMT_N)\le\r_Y(\SMT_{N'}).
 \end {equation}
Then
 $$
m(X,\r_X)\le m(Y,\r_Y).
 $$
 \end {prop}

 \begin {proof}
Let $N\ss X$ be an arbitrary finite set. As in the proof of
Proposition~\ref{prop:project}, we have:
 $$
m(X,\r_X)=\inf_{\{N:N\ss X\}}\frac{\r_X(\SMT_N)}{\r_X(\MST_N)}=
\inf_{\{N':N'\ss Y\}}\inf_{\{N:f(N)=N'\}}
                  \frac{\r_X(\SMT_N)}{\r_X(\MST_N)}.
 $$
Since $f$ does not increase distances, then
$\r_X(\MST_N)\ge\r_Y(\MST_{f(N)})$ (see Proposition~\ref{prop:bundle}); on
the other hand, due to our assumption, there exists a sequence of finite
sets $N_i\ss X$, $f(N_i)=N'$, such that
$\r_X(\SMT_{N_i})\le\r_Y(\SMT_{N'})+\e_i$, where the sequence of positive
numbers $\e_i$ tends to $0$ as $i\to\infty$, and the sequence of positive
numbers $\r_X(\SMT_{N_i})$ tends to $\inf_{\{N:f(N)=N'\}}\r_X(\SMT_N)$.
Therefore,
 $$
\frac{\r_X(\SMT_{N_i})}{\r_X(\MST_{N_i})}\le
\frac{\r_Y(\SMT_{N'})+\e_i}{\r_Y(\MST_{N'})},
 $$
and, taking in account that $\{N_i\}\subset\{N:f(N)=N'\}$, we get:
 \begin {multline*}
\inf_{\{N':N'\ss Y\}}\inf_{\{N:f(N)=N'\}}\frac{\r_X(\SMT_N)}{\r_X(\MST_N)}
\le\\
\inf_{\{N':N'\ss Y\}}
    \inf_{\{N_i\}}\frac{\r_X(\SMT_{N_i})}{\r_X(\MST_{N_i})}
\le
\inf_{\{N':N'\ss Y\}}\inf_i\frac{\r_Y(\SMT_{N'})+\e_i}{\r_Y(\MST_{N'})}=\\
\inf_{\{N':N'\ss Y\}}\frac{\r_Y(\SMT_{N'})}{\r_Y(\MST_{N'})}=
m(Y,\r_Y).
 \end {multline*}
The proof is complete.
 \end {proof}

Let $M$ be an arbitrary connected $n$-dimensional Riemannian manifold. For
each piecewise-smooth curve $\g$ by $\len(\g)$ we denote the length of
$\g$ with respect to the Riemannian metric. By $\r$ we denote the
intrinsic metric generated by the Riemannian metric. We recall that
 $$
\r(x,y)=\inf_\g\len(\g),
 $$
where the greatest lower bound is taken over all piecewise-smooth curves
$\g$ joining the points $x$ and $y$.

Let $P$ be a point from $M$. We consider the normal coordinates
$(x^1,\ldots,x^n)$ centered at $P$, such that the Riemannian metric
$g_{ij}(x)$ calculated at $P$ coincides with $\dl_{ij}$. Let $U(\dl)$ be
the open convex ball centered at $P$ and having the radius $\dl$. Any two
points $x$ and $y$ from the ball are joined by a unique geodesic $\g$
lying in $U(\dl)$.  At that time, $\r(x,y)=\len(\g)$.  Thus, the ball
$U(\dl)$ is a metric space with intrinsic metric, that is, the distance
between the points equals to the greatest lower bound of the curves`
lengths over all the measurable curves joining the points. Notice that in
terms of the coordinates $(x^i)$ the ball $U(\dl)$ is defined as follows:
 $$
U(\dl)=\bigl\{(x^1)^2+\cdots+(x^n)^2<\dl^2\bigr\}.
 $$
Therefore, if we define the Euclidean distance $\r_e$ in $U(\dl)$ (in
terms of the normal coordinates $(x^i)$), then the metric space
$\bigl(U(\dl),\r_e\bigr)$ also is the space with intrinsic metric
generated by the Euclidean metric $\dl_{ij}$.

Since the Riemannian metric $g_{ij}(x)$ depends on $x\in U(\e)$ smoothly,
then for any $\e$, $1/n^2>\e>0$, there exists a $\dl>0$, such that
 \begin {equation}\label{eq:small}
|g_{ij}(x)-\dl_{ij}|<\e
 \end {equation}
for all points $x\in U(\dl)$. The latter implies the following
Proposition.

 \begin {prop}
Let $\|v\|_g$ be the length of the tangent vector $v\in T_xM$ with respect
to the Riemannian metric $g_{ij}$, and let $\|v\|_e$ be the length of $v$
with respect to the Euclidean metric $\dl_{ij}$.  If for any $i$ and $j$
the inequality~(\ref{eq:small}) holds, then
 $$
\sqrt{1-n^2\e}\,\|v\|_e\le\|v\|_g\le\sqrt{1+n^2\e}\,\|v\|_e.
 $$
 \end {prop}

 \begin {proof}
Consider an orthogonal transformation (with respect to the Euclidean
metric $\dl_{ij}$) reducing the matrix $(g_{ij})$ to the diagonal form
$\diag(\l_1,\ldots,\l_n)$, and let $(c^i_j)$ be the matrix of this
transformation. Then $\l_k=\sum_{i,j}c^i_kc^j_kg_{ij}$, therefore, using
that $|c^i_j|\le1$ due to orthogonality of $(c^i_j)$, we get:
 \begin {multline*}
|\l_k-1|=\Bigl|\sum_{i,j}(c^i_kc^j_kg_{ij}-c^i_kc^j_k\dl_{ij})\Bigr|\le\\
\sum_{i,j}|c^i_k|\cdot|c^j_k|\cdot|g_{ij}-\dl_{ij}|\le
\sum_{i,j}|g_{ij}-\dl_{ij}|\le n^2\e.
 \end {multline*}
So we have:
 $$
\|v\|_g=\sqrt{\sum_k\l_kv^kv^k}\le\sqrt{\max_k\l_k\,\sum_kv^kv^k}
\le\sqrt{1+n^2\e}\|v\|_e.
 $$
Similarly, we get
 $$
\|v\|_g\ge\sqrt{1-n^2\e}\|v\|_e.
 $$
The proof is complete.
 \end {proof}

Using the definition of the distance between a pair of points of a
connected Riemannian manifold, we obtain the following result.

 \begin {cor}\label{cor:riem}
Let $M$ be an arbitrary connected $n$-dimensional Riemannian manifold, and
let $U(\dl)$, $\r$, and $\r_e$ be as above. Then for an arbitrary $\e$,
$1/n^2>\e>0$,  there exists a $\dl>0$, such that
 $$
\sqrt{1-n^2\e}\,\r_e(x,y)\le\r(x,y)\le\sqrt{1+n^2\e}\,\r_e(x,y)
 $$
for all points $x,y\in U(\dl)$.
 \end {cor}

Since the Steiner ratio is evidently the same for any convex open subsets
of $\R^n$,  Corollary~\ref{cor:riem} and Lemma~\ref{lem:equiv_metr} lead
to the following result.

 \begin {cor}\label{cor:steiner_for_manifolds}
Let $M$ be an arbitrary $n$-dimensional Riemannian manifold, let
$U(\e)\ss M$ be an open convex ball of a small radius $\e$, and let $P$ be
the center of $U(\e)$. By $\r$ we denote the metric on $M$ generated by
the Riemannian metric. Then
 $$
\sqrt{\frac{1-n^2\e}{1+n^2\e}}m(\R^n)\le m\bigl(U(\e),\r\bigr)\le
\sqrt{\frac{1+n^2\e}{1-n^2\e}}m(\R^n),
 $$
where $m(\R^n)$ stands for the Steiner ratio of the Euclidean space
$\R^n$.
 \end {cor}

Now let us prove the main theorems stated in Introduction.

 \begin {Proof}{Proof of Theorem~\ref{th:riem}}
Let $M$ be an arbitrary connected $n$-dimensional Riemannian manifold, and
let $\r$ be the metric generated by the Riemannian metric of $M$. Let
$P\in M$ be an arbitrary point from $M$, and let $U(\e)$ be an open convex
ball centered at $P$ and having radius $\e<1/n^2$. As above, let $(x^i)$
be normal coordinates on $U(\e)$, and let  $\r_e$ be the metric on $U(\e)$
generated by the Euclidean metric $\dl_{ij}$ (with respect to $(x^i)$).

For some decreasing sequence $\{\e_i\}$ of positive numbers with $\e_i<\e$
for any $i$, where $\e_i\to0$ as $i\to\infty$, we consider a family of
nested subsets $X_i=U(\e_i)$. Notice that due to convexity of Euclidean
balls $\bigl(U(\e),\r_e\bigr)$ we have:
 $$
m\bigl(U(\e),\r_e\bigr)=m(\R^n).
 $$
Besides, due to convexity of the balls $U(\e)$ with respect to the
intrinsic metric $\r'$ generated by the Riemannian metric $g_{ij}$, this
intrinsic metric $\r'$ coincides with the restriction of the metric $\r$.
Thus, the ball $U(\e)$ with the intrinsic metric $\r'$ is a subspace in
$(M,\r)$.

Corollary~\ref{cor:steiner_for_manifolds} implies that
 $$
m(X_i,\r)\le\sqrt{\frac{1+n^2\e}{1-n^2\e}}m(\R^n).
 $$
Since $\sqrt{\frac{1+n^2\e}{1-n^2\e}}\to1$ as $i\to\infty$ due to the
choice of $\{\e_i\}$, we get
 $$
\inf_im(X_i,\r)\le m(\R^n).
 $$
But, due to Lemma~\ref{lem:subspace} we have:
 $$
m(M,\r)\le\inf_im(X_i,\r).
 $$
The proof is complete.
 \end {Proof}

 \begin {Proof}{Proof of Theorem~\ref{th:bundle}}
Let $\pi\:W\to M$ be a locally isometric covering, where $W$ and $M$ are
connected Riemannian manifolds.  By $\r_W$ and $\r_M$ we denote the
metrics generated by the Riemannian metrics on $W$ and $M$, respectively.
Notice that a locally isometric covering does not increase distances,
since the image of a measurable curve $\g$ has the same length as $\g$
has.

We consider an arbitrary finite set $N'\ss M$. Let $G'_i$ be a family of
trees on finite sets $\bN'_i\sp N'$ such that
 $$
\r_M(G'_i)\to\r_M(\SMT_{N'})\ \ \text{as $i\to\infty$}.
 $$
For each $G'_i$ by $\G'_i$ we denote an embedded network of the type
$G'_i$ on $M$ such that the vertex set of $\G'_i$ is $V(G'_i)$ and the
length of $\G'_i$ differs from $\r_M(G'_i)$ at most by $1/i$.  Let $\G_i$
be a connected component of $\pi^{-1}(\G'_i)$, and
$N_i=\pi^{-1}(N)\cap\G_i$. Since the network $\G'_i$ is contractible, then
the restriction of the fibration $\pi$ onto $\G'_i$ is trivial. Therefore
the restriction of the projection $\pi$ onto $\G_i$ is a homeomorphism.
Since the projection $\pi$ is locally isometric, then the length of the
network $\G_i$ in $W$ coincides with the length of the network $\G'_i$ in
$M$.  But $\r_W(\SMT_{N_i})$ does not exceed the length of $\G_i$,
therefore
 $$
\r_W(\SMT_{N_i})\le\r_M(\SMT_{N'})+\e_i,
 $$
where the sequence $\{\e_i\}$ of positive numbers tends to $0$ as
$i\to\infty$. So,
 $$
\inf_{\{N:f(N)=N'\}}\r_W(\SMT_N)\le\r_M(\SMT_{N'}).
 $$
It remains to apply Proposition~\ref{prop:project1}. The proof is
complete.
 \end {Proof}

 \begin {Proof}{Proof of Corollary~\ref{cor:2-dim}}
It follows from Theorems~\ref{th:riem}, and~\ref{th:bundle}; Du and Hwang
theorem~\cite{DuHwang} saying that the Steiner ratio of the Euclidean
plane equals $\sqrt3/2$; and also from Rubinstein and Weng
theorem~\cite{Rubin} saying that the Steiner ratio of the standard two
dimensional sphere with constant positive curvature metric equals
$\sqrt3/2$.
 \end {Proof}

 \begin {Proof}{Proof of Theorem~\ref{th:hyper}}
Let us consider the Poincar\'e model of the Lobachevsky plane $L^2(-1)$
with constant curvature $-1$. We recall that this model is a radius $1$
flat disk centered at the origin of the Euclidean plane with Cartesian
coordinates $(x,y)$, and the metric $ds^2$ in the disk is defined as
follows:
 $$
ds^2=4\frac{dx^2+dy^2}{(1-x^2-y^2)^2}.
 $$

It is well known that for each regular triangle in the Lobachevsky plane
the circumscribed circle exists. The radii emitted out of the center
of the circle to the vertices of the triangle forms the angles of $120^\c$.

Let $r$ be the radius of the circumscribed circle. The cosine rule implies
that the length $a$ of the side of the regular triangle can be calculated
as follows:
 $$
\ch a=\ch^2r-\sh^2r\cos\frac{2\pi}3=1+\frac32\sh^2r.
 $$
It is easy to verify that for such triangle the length of $\MST$ equals
$2a$, and the length of $\SMT$ equals $3r$. Therefore, the Steiner ratio
$m(r)$ for the regular triangle inscribed into the circle of radius $r$ in
the Lobachevsky plane $L^2(-1)$ has the form
 $$
m(r)=\frac32\cdot\frac{r}{\arcch\bigl(1+\frac32\sh^2(r)\bigr)}.
 $$
It is easy to calculate that limit of the function $m(r)$ as $r\to\infty$
is equal to $3/4$. The proof is complete.
 \end {Proof}

 \begin {Proof}{Proof of Theorem~\ref{th:hypersurf}}
It is easy to see that the Taylor series for the function $m(r)$ at $r=0$
has the following form:
 $$
\frac{\sqrt3}2-\frac{r^2}{16\sqrt3}+O(r^4).
 $$
Therefore, $m(r)$ is strictly less than $\sqrt3/2$ in some interval
$(0,\e)$. The latter means that for sufficiently small regular triangles
on the surfaces of constant curvature $-1$, the relation of the lengths of
$\SMT$ and $\MST$ is strictly less than $\sqrt3/2$. The proof is complete.
 \end {Proof}

\bibliographystyle{plain}

 \end {document}